\documentclass[11pt]{amsart}
\usepackage{amsmath,amssymb,graphicx,mathtools}
\usepackage[unicode]{hyperref}
\usepackage[margin=1in]{geometry}  
\usepackage{color}
\usepackage{enumerate}

\newtheorem{prop}{Proposition}
\newtheorem{thm}[prop]{Theorem}
\newtheorem{lemma}[prop]{Lemma}
\theoremstyle{definition}
\newtheorem{remark}[prop]{Remark}

\newcommand{\Aut}{\operatorname{Aut}}
\newcommand{\Out}{\operatorname{Out}}
\newcommand{\Inn}{\operatorname{Inn}}
\newcommand{\Art}{A}

\newcommand{\Mod}{\operatorname{Mod}}
\renewcommand{\C}{\mathcal{C}}

\newcommand{\Z}{\mathbb{Z}}
\newcommand{\id}{\operatorname{id}}

\newcommand{\prodd}{\operatorname{prod}}

\makeatletter
\@namedef{subjclassname@2020}{\textup{2020} Mathematics Subject Classification}
\makeatother

\begin{document}
\author{Matthieu Calvez}
\address{IRISIB, Rue Royale 150, 1000 Bruxelles, Belgique}
\address{Instituto de matem\'aticas\\ 
Universidad de Valpara\'iso\\ 
Gran Breta\~na 1091\\ 3er piso\\ 
Playa Ancha\\ Valpara\'iso\\ Chile}
\email{calvez.matthieu@gmail.com}

\author{Ignat Soroko}
\address{Department of Mathematics\\
Florida State University\\
Tallahassee\\ FL 32306\\ USA}
\email{soroko@math.fsu.edu, ignat.soroko@gmail.com}

\title{Property $R_\infty$ for some spherical and affine Artin--Tits groups}


\begin{abstract}
Let $n\geqslant2$. In this note we give a short uniform proof of property $R_\infty$ for 
the Artin--Tits groups of spherical types $A_n$, $B_n$, $D_4$, $I_2(m)$ ($m\geqslant3$), their 
pure subgroups, and for the Artin--Tits groups of affine types $\widetilde A_{n-1}$ and $\widetilde C_n$. 
In particular, we provide an alternative proof of a recent result of 
Dekimpe, Gon\c{c}alves and Ocampo, who established property $R_{\infty}$ for pure Artin braid groups. 
\end{abstract}

\date{\today}

\maketitle

\section{Introduction}

Let $G$ be a group and $\varphi$ an automorphism of $G$. An equivalence relation on $G$ is defined by saying that $g,h\in G$ are \emph{$\varphi$-twisted conjugate} if and only if there exists $x\in G$ such that $h=xg \varphi(x)^{-1}$. 
The number of its equivalence classes is the \emph{Reidemeister number} of $\varphi$, denoted by $R(\varphi)$. We say that~$G$ \emph{has property $R_{\infty}$} if $R(\varphi)=\infty$ for all 
$\varphi\in \Aut(G)$. 

The notion of twisted conjugacy classes arises in Nielsen fixed point theory, where under certain natural conditions the Reidemeister number serves as an upper bound for a homotopy invariant called the Nielsen number. Also the twisted conjugacy classes appear naturally in Selberg theory and in some topics of algebraic geometry. See the introduction sections in~\cite{TW11}, \cite{Romankov} and references therein.

Property $R_{\infty}$ has been proved for a number of groups (see \cite{FT15} or \cite{FN16} for a detailed list) among which are Artin braid groups~\cite{FGD10}, some classes of large Artin--Tits groups~\cite{Juhasz}, and, very recently, pure Artin braid groups~\cite{DGO21} and some right-angled Artin--Tits groups~\cite{DS21}.

Recall that a \textit{Coxeter matrix} over a finite set $S$ is a symmetric matrix $(m_{st})_{s,t\in S}$ with entries in $\{1,2,\dots,\infty\}$, such that $m_{ss}=1$ for all $s\in S$ and $m_{st}=m_{ts}\geqslant2$ if $s\ne t$. A Coxeter matrix can be encoded by the corresponding \textit{Coxeter graph} $X$ having $S$ as the set of vertices. Two distinct vertices $s,t\in S$ are connected with an edge in $X$ if $m_{st}\geqslant 3$, and this edge is labeled with $m_{st}$ if $m_{st}\geqslant 4$. The \textit{Artin--Tits group of type $X$} is the group $\Art(X)$ given by the presentation:
\[
\Art(X)=\langle S\mid \prodd(s,t,m_{st})=\prodd(t,s,m_{ts}), \text{ for all } s\ne t,\,m_{st}\ne\infty\rangle,
\]
where $\prodd(s,t,m_{st})$ is the word $stst\dots$ of length $m_{st}\geqslant2$. 
The \textit{Coxeter group $W(X)$ of type $X$} is the quotient of $\Art(X)$ by all relations of the form $s^2=1$, $s\in S$. The most famous example is the Artin braid group on $n$ strands, or Artin--Tits group $\Art(A_{n-1})$, whose corresponding Coxeter group is the symmetric group on $n$ elements. 

The groups $\Art(X)$ and $W(X)$ are called \emph{irreducible} if $X$ is connected. An Artin--Tits group $\Art(X)$ is of \emph{spherical type} if the corresponding Coxeter group $W(X)$ is finite, and of \emph{affine type} if $W(X)$ acts geometrically (i.e.\ properly discontinuously and cocompactly by isometries) on a euclidean space. 

We denote $P(X)$ the kernel of the natural epimorphism $\Art(X)\to W(X)$. It is called the \emph{pure Artin--Tits group of type $X$}. If $X$ is of spherical type, the group $P(X)$ has finite index in $\Art(X)$. The group $P(A_{n-1})$ is also called the \emph{pure Artin braid group on $n$ strands}.

\begin{figure}
\begin{center}
\includegraphics[scale=1.2]{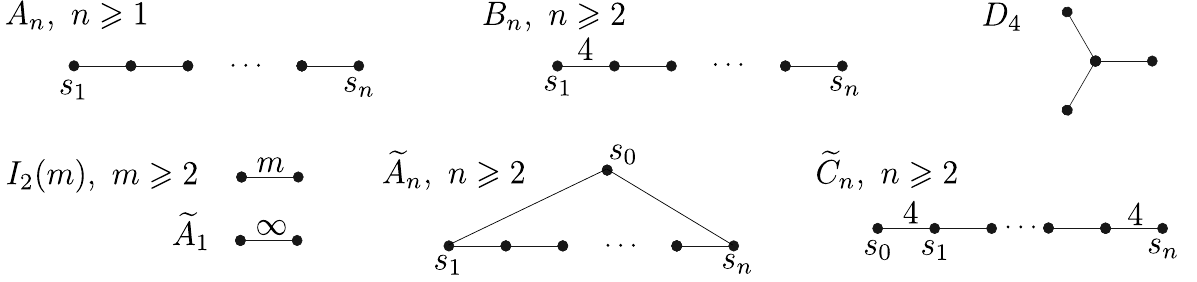}
\end{center}
\caption{The defining Coxeter graphs for the Artin--Tits groups under consideration.}
\label{F:Coxeter}
\end{figure}

In this note we focus on irreducible Artin--Tits groups associated to the Coxeter graphs depicted in Figure~\ref{F:Coxeter} and on the pure Artin--Tits groups corresponding to those of them which are of spherical type. For a group $G$, we denote by $\overline{G}=G/Z(G)$ its quotient by the center. We prove the following theorem.

\begin{thm}\label{T:Main} 
Let $n\geqslant2$. 
\begin{enumerate}
\item Let $X\in \{A_n, B_n, D_4, I_2(m)\, 
(m\geqslant3,m\ne\infty)\}$. 
Then the Artin--Tits group $A(X)$, the pure Artin--Tits group $P(X)$, and their central quotients $\overline{\Art(X)}$ and $\overline{P(X)}$ have property~$R_{\infty}$. 
\item Let $X\in \{\widetilde A_{n-1},\widetilde C_n\}$. Then the Artin--Tits group $A(X)$ has property $R_{\infty}$. 
\end{enumerate}
\end{thm}

As is well-known (see for example \cite{Romankov}), free abelian groups do not have property $R_\infty$; therefore $\Art(A_1)\simeq P(A_1)\simeq\Z$ and $\Art(I_2(2))\simeq P(I_2(2))\simeq \Z^2$ do not have property~$R_{\infty}$.

Our proof is uniform and is an application of the strategy implemented in~\cite{FGD10} for $\Art(A_n)$ to all other groups of Theorem~\ref{T:Main} (except 
$A(I_2(m)),  P(I_2(m))$, and $A(\widetilde A_1)$, which are virtually free). 
Let~$G$ be one of these groups and let $\Gamma=G/Z(G)$. We use the fact that the groups $\Gamma$ and $\Aut(\Gamma)$ are finite index subgroups of the extended mapping class group of a suitable punctured surface to deduce that the ``mapping torus" of any automorphism $\varphi$ of $\Gamma$ acts non-elementarily on the curve complex of that surface. This, combined with a result of Delzant (Lemma~\ref{L:Delzant}) and Proposition~\ref{P:Cosets}, allows to distinguish infinitely many twisted conjugacy classes of $\varphi$.

In particular, we obtain a short independent proof of a recent theorem of Dekimpe--Gon\c{c}alves--Ocampo \cite{DGO21} who established property $R_\infty$ for the pure Artin braid groups $P(A_n)$. For all other groups in Theorem~\ref{T:Main} (except $A(A_n)$ and $P(A_n)$) our result is new, to the best of our knowledge.

\subsection*{Acknowledgments}
The authors are grateful to Karel Dekimpe for his encouragement to write this note and to Peter Wong for a useful discussion. The authors would also like to thank the anonymous referee for his or her valuable suggestions which improved the quality of the paper. The first author was supported by the EPSRC New Investigator Award EP/S010963/1 and FONDECYT 1180335. The second author acknowledges support from the AMS--Simons Travel Grant. 

\section{Preliminary results}\label{S:Sketch}

In this section we collect preliminary results which we will use to prove Theorem~\ref{T:Main} in Section~\ref{S:Proof}. 

Recall that the center of an irreducible Artin--Tits group $A(X)$ (and its corresponding pure Artin--Tits group $P(X)$) of spherical type is cyclic, see~\cite[Th.\,7.2]{BrieskornSaito} (and \cite[Cor.\,7]{CP19}, respectively). On the other hand, the center of an irreducible Artin--Tits group of affine type is trivial, see~\cite[Prop.\,11.9]{McCammondSulway}. 
It is known (and is easy to prove) that the central quotient $\overline{\Art(X)}$ has trivial center itself, which is also true for $\overline{P(X)}$~\cite[Cor.\,7]{CP19}.

The following Proposition allows to reduce the proof of Theorem~\ref{T:Main} about $\Art(X)$ and $P(X)$ to proving a similar statement for their central quotients $\overline{\Art(X)}$ and $\overline{P(X)}$, respectively. 

\begin{prop}[See e.g.~\protect{\cite[(2.2)]{FGD10}}]\label{prop:epi}
Let $G$ be a group, $\varphi$ be an automorphism of $G$, and $H$ be a normal $\varphi$-invariant subgroup of $G$. Denote by $\bar\varphi$ the automorphism induced by $\varphi$ on $G/H$. Then the condition $R(\bar\varphi)=\infty$ implies $R(\varphi)=\infty$. In particular, if $G/Z(G)$ has property $R_{\infty}$, then so does $G$. \qed 
\end{prop}

Recall that for a group $G$, $\Inn(G)$ denotes the subgroup of $\Aut(G)$ consisting of all inner automorphisms. For $g\in G$, we denote by $\iota_g$ the inner automorphism  $x\mapsto gxg^{-1}$. The map $\iota: g\mapsto \iota_g$ identifies $\Inn(G)$ with $G/Z(G)$ and $\Inn(G)$ is a normal subgroup of $\Aut(G)$. The \emph{outer automorphism group} of $G$ is the quotient $\Out(G)=\Aut(G)/\Inn(G)$. 

Now, let us review an interesting relationship between the twisted conjugacy classes of automorphisms of a centerless group $G$ and the actual conjugacy classes in an extension of $G$. 
Let $\varphi$ be an automorphism of a centerless group~$G$, let $m\in\{1,2,3,\dots,\infty\}$ 
be the order of $\varphi$ in $\Out(G)$; if $m<\infty$, let $p\in G$ be the unique element such that $\varphi^m=\iota_p$. 
Consider the group
\[
G_\varphi=
\begin{cases} 
G*\langle t\rangle/\langle tgt^{-1}=\varphi(g) \text{ for all }g\in G\,\rangle,  & \text{if $m=\infty$,}\\
G*\langle t\rangle/\langle tgt^{-1}=\varphi(g) \text{ for all }g\in G,\,t^m=p\,\rangle,  & \text{if $m<\infty$.} 
\end{cases}
\]
Note that if $m=\infty$, $G_\varphi$ is the semidirect product $G\rtimes_\varphi\Z$, and if $m=1$, $G_\varphi$ is isomorphic to $G$.

\begin{lemma}\label{L:Embeds}
Let $G$ be a centerless group, $\varphi\in\Aut(G)$, and $G_\varphi$ defined as above. Then $G$ is a normal subgroup of $G_\varphi$, the quotient $G/G_{\varphi}$ is cyclic, and the assignment $g\mapsto \iota_g$ and $t\mapsto \varphi$ defines an injective homomorphism from $G_{\varphi}$ to $\Aut(G)$. 
\end{lemma}
\begin{proof}
We prove the second statement first. Indeed, the given assignment $g\mapsto \iota_g$ and $t\mapsto \varphi$ defines a homomorphism since relations in $G_{\varphi}$ hold in $\Aut(G)$: the conjugation by $\varphi$ in $\Aut(G)$ sends the inner automorphism $\iota_g$ for $g\in G$ to $\varphi \iota_g\varphi^{-1} = \iota_{\varphi(g)}$.
Every element of $G_{\varphi}$ can be written as $t^kg$, for some $g\in G$ and 
$k\in\Z$ if $m=\infty$ and $k\in\{0,\ldots,m-1\}$ if $m<\infty$. Then the element $t^kg$ is sent to $\varphi^k \iota_g$, and if $\varphi^k\iota_g=\id_G$, we deduce first $\varphi^k= {\iota_g}^{-1}=\iota_{g^{-1}}$, which forces $k=0$ by definition of $m$. Then since $G$ is assumed to be centerless, $\iota_g=\id_G$ implies that $g=1$ and hence the homomorphism in question is injective.

If $m=\infty$, then $G_\varphi$ is the semidirect product $G\rtimes_\varphi\langle t\rangle$, hence $G$ is a subgroup of $G_\varphi$. If $m<\infty$, we notice that the injective homomorphism $\iota\colon G\to\Aut(G)$, $g\mapsto \iota_g$, factors through the composition $G\to G_\varphi\to\Aut(G)$, and hence the map $G\to G_\varphi$ is itself injective. 
In either case, $G$ is a normal subgroup and $G_{\varphi}/G$ is generated by the coset $tG$.
\end{proof}

\begin{prop}[cf.\,\protect{\cite[Lemma 6.2]{FGD10}}]\label{P:Cosets}
Let $G$ be a centerless group, $\varphi\in\Aut(G)$, and $G_\varphi$ defined as above. Two elements $g,h\in G$ are $\varphi$-twisted conjugate if and only if the elements $gt$ and~$ht$ of $G_{\varphi}$ are conjugate in $G_{\varphi}$. In particular, 
$R(\varphi)=\infty$ if and only if the coset $Gt$ in $G_{\varphi}$ contains infinitely many conjugacy classes.
\end{prop}

\begin{proof}
Suppose that $g,h\in G$ are $\varphi$-twisted conjugate. Then there exists some $x\in G$ such that $h=xg\varphi(x)^{-1}$. In $G_{\varphi}$, we then have
$h=xgtx^{-1}t^{-1}$, which is to say that $ht=x(gt)x^{-1}$. Conversely, suppose that $gt$ and $ht$ are conjugate in $G_{\varphi}$. Then there is some 
element $t^kx$, with $x\in G$ and $k\in\Z$ if $m=\infty$, and $k\in \{0,\ldots,m-1\}$ if $m<\infty$, such that $ht=t^k x\cdot gt \cdot x^{-1}t^{-k}$, or $ht^{k+1}x = t^k xgt$. 
Using the relation $t^k \gamma = \varphi^k(\gamma)t^k$ for $\gamma\in G$, we obtain 
$h\varphi^{k+1}(x)t^{k+1}= \varphi^k(xg)t^{k+1}$. Canceling $t^{k+1}$ on the right and setting $y=\varphi^k(x)$, we obtain 
$h\varphi(y)=y \varphi^k(g)$, or $h = y \varphi^k(g)\varphi(y)^{-1}$. This is to say that $h$ and $\varphi^k(g)$ are $\varphi$-twisted conjugate. To conclude, it suffices to know that 
$\varphi^k(g)$ is $\varphi$-twisted conjugate to $g$: in general, for every automorphism $\psi$, any $\gamma\in G$ is $\psi$-twisted conjugate to its image $\psi(\gamma)$ according to the equality
$\gamma = \gamma\psi(\gamma)\psi(\gamma)^{-1}$. 
\end{proof}

In particular, we recover from Proposition~\ref{P:Cosets} the well-known fact that for any inner automorphism $\varphi$, the number of $\varphi$-twisted conjugacy classes $R(\varphi)$ equals the number of conjugacy classes in $G$ (see e.g.\,\cite[Lemma~1.3]{Tro19} for $\varphi=\id$).

Let $G$ be a group acting on a Gromov hyperbolic space $S$. For any $s\in S$ consider the set~$\Lambda(G)$ of accumulation points of the orbit $Gs$ on the boundary $\partial S$. The action of $G$ on $S$ is called \emph{non-elementary}, if $\Lambda(G)$ contains at least three points, see~\cite[Sec.\,3]{Osi16}.

The main tool for our proof is the following result due to Delzant: 

\begin{lemma}[\protect{\cite[Lemma 6.3]{FGD10},\cite[Lemma 3.4]{LevittLustig}}]\label{L:Delzant}
Let $\Gamma$ be a group acting non-elementarily on a Gromov hyperbolic space, let $K$ be a normal subgroup of $\Gamma$ such that the quotient $\Gamma/K$ is abelian. Then any coset of $K$ contains infinitely many conjugacy classes. \qed 
\end{lemma}

This lemma will be used as follows. Let $G$ be a centerless group, $\varphi\in \Aut(G)$ and $G_{\varphi}$ as above, then if $G_{\varphi}$ acts non-elementarily on a Gromov hyperbolic space, 
Lemma \ref{L:Delzant} (with $\Gamma=G_{\varphi}$ and $K=G$) and Proposition \ref{P:Cosets} show that $R(\varphi)=\infty$. 
We will also make use of the following result of Fel'shtyn: 

\begin{thm}[\protect{\cite[Th.\,3]{Fel}}]\label{T:Hyp}
Every non-elementary hyperbolic group has property $R_{\infty}$. \qed 
\end{thm}

\section{Proof of Theorem~\ref{T:Main}}\label{S:Proof}

\begin{prop}\label{prop:aut}
Let $n\geqslant 3$ and let $X\in \{A_n, B_n, D_4, \widetilde A_{n-1}, \widetilde C_{n-1}\}$. Let $G$ be either the Artin--Tits group $\Art(X)$, or the pure Artin--Tits group $P(A_n)$, $P(B_n)$, $P(D_4)$, and denote $\Gamma=G/Z(G)$. Then there exists a surface $\Sigma_\Gamma$ with punctures such that both $\Gamma$ and $\Aut(\Gamma)$ can be embedded as nested subgroups of finite index in the extended mapping class group of $\Sigma_\Gamma$: 
\[
\Gamma\leqslant\Aut(\Gamma)\leqslant\Mod^{\pm}(\Sigma_\Gamma).
\]
Here $\Gamma$ is identified with the subgroup $\Inn(\Gamma)$ of inner automorphisms. If $g$ and $p$ denote the genus and the number of punctures of $\Sigma_\Gamma$, respectively, then they satisfy the inequality: $3g+p-4>0$.
\end{prop}
\begin{proof}
It is shown in~\cite[Sec.\,2]{CC05} that for $G=\Art(X)$ with $X\in \{A_n, B_n, \widetilde A_{n-1}, \widetilde C_{n-1}\}$, the group $\Gamma=G/Z(G)$ is a subgroup of finite index in the extended mapping class group of a sphere with $n+2$ punctures $\Sigma_\Gamma=S_{n+2}$. Namely, $\overline{\Art(A_{n})}$, $\overline{\Art(B_{n})}$, and ${\Art(\widetilde C_{n-1})}$ are the subgroups of all orientation preserving mapping classes of $S_{n+2}$ that fix $1$, $2$, and $3$ punctures of $S_{n+2}$, respectively. Similarly, ${\Art(\widetilde A_{n-1})}$ is a subgroup of index $n$ in $\overline{\Art(B_n)}$. For $G=P(A_n)$, it is briefly mentioned in~\cite[p.\,324]{CC05} and discussed in detail in~\cite[9.3]{FM12} that $\Gamma=G/Z(G)$ is the group of all orientation preserving mapping classes of $S_{n+2}$ that fix all punctures pointwise, i.e.\ it is the pure orientation preserving mapping class group of the sphere with $n+2$ punctures, which has index $2(n+2)!$ in $\Mod^{\pm}(S_{n+2})$. In general, if $G=P(X)$ is the pure Artin--Tits group for a connected Coxeter graph $X$ of spherical type, it was proven in~\cite[Cor.\,7]{CP19} that $G/Z(G)$ is isomorphic to a subgroup of finite index in $\overline{\Art(X)}$. In particular, for $G=P(B_n)$, the group $\Gamma=G/Z(G)$ is isomorphic to a subgroup of finite index in $\overline{\Art(B_n)}$, and hence to a subgroup of finite index in $\Mod^{\pm}(S_{n+2})$. Notice that for $\Sigma_\Gamma=S_{n+2}$ we have $g=0$, $p=n+2$, so the required inequality becomes: $3g+p-4=n-2>0$, which is true if $n\geqslant3$.

Now we invoke the Corollary to the Ivanov--Korkmaz theorem~\cite[Cor.\,4]{CC05}, and conclude that in all the above cases the group $\Aut(\Gamma)$ is isomorphic to the normalizer of $\Gamma$ in $\Mod^{\pm}(S_{n+2})$. Since $\Aut(\Gamma)$ contains $\Gamma$ as a subgroup (identified with the subgroup $\Inn(\Gamma)$ of all inner automorphisms), and $\Gamma$ has finite index in $\Mod^{\pm}(S_{n+2})$, we conclude that $\Aut(\Gamma)$ has itself finite index in $\Mod^{\pm}(S_{n+2})$. 

If $G=\Art(D_4)$, let $\Gamma=\overline{\Art(D_4)}$ and $\Sigma_\Gamma$ be the torus with three punctures. It was shown in~\cite[Th.\,1]{Sor20a} that $\Gamma$ is isomorphic to the pure orientation preserving mapping class group of $\Sigma_\Gamma$ and in~\cite[Cor.\,6]{Sor20} that $\Aut(\Gamma)$ is isomorphic to the extended mapping class group $\Mod^{\pm}(\Sigma_\Gamma)$, which contains $\Gamma$ as a subgroup of index $12$. For $G=P(D_4)$, we refer to~\cite[Cor.\,7]{CP19} again to conclude that $\Gamma=G/Z(G)$ is isomorphic to a subgroup of finite index in $\overline{\Art(D_4)}$, and hence to a subgroup of finite index in $\Mod^{\pm}(\Sigma_\Gamma)$, with $\Sigma_\Gamma$ being the torus with three punctures. Clearly, the required inequality is satisfied for $\Sigma_\Gamma$ in this case as well: $3g+p-4=3+3-4>0$.
\end{proof}

Now we are ready to prove Theorem~\ref{T:Main}. 

\begin{proof}[Proof of Theorem~\ref{T:Main}]
Let $G$ be any of the groups from the theorem; we are going to prove that its central quotient $G/Z(G)$ has property $R_\infty$. Then Proposition~\ref{prop:epi} will yield the desired conclusion for~$G$.

First, let $X=I_2(m)$, $m\geqslant 3$, $m\ne\infty$. 
It is known that $\overline{\Art(X)}$ is isomorphic to $C_m\star C_2$ if $m$ is odd, and to $C_{m/2}\star \Z$ if $m$ is even, where $C_n$ denotes the cyclic group of order $n$, see e.g.~\cite[Sec.\,5]{CP02}. In any case, $\overline{\Art(X)}$ contains a finite index nonabelian free subgroup and hence is non-elementary hyperbolic. By~\cite[Cor.\,7]{CP19}, the central quotient $\overline{P(X)}$ is isomorphic to a subgroup of finite index in $\overline{A(X)}$, hence also contains a nonabelian free subgroup of finite index, and thus is also non-elementary hyperbolic. (One can actually prove that $\overline{P(I_2(m))}\cong F_{m-1}$, the free group of rank $m-1$, see~\cite[Proof of Lemma\,6.2]{CP19}.) 
Theorem~\ref{T:Hyp} now shows that $\overline{\Art(X)}$ and $\overline{P(X)}$ have property $R_{\infty}$, and hence, by Proposition~\ref{prop:epi}, $\Art(X)$ and $P(X)$ have property~$R_{\infty}$ as well. 

In view of the isomorphisms $A(A_2)\simeq A(I_2(3))$ and $A(B_2)\simeq A(I_2(4))$, the above reasoning covers the groups $A(A_2)$, $A(B_2)$, $P(A_2)$, $P(B_2)$; 
the result for $\Art(\widetilde A_{1})\simeq F_2$ also follows from Theorem~\ref{T:Hyp}.

Suppose now that $n\geqslant3$ and $G=\Art(X)$ for $X\in \{A_n, B_n, D_4, \widetilde A_{n-1}, \widetilde C_{n-1}\}$, or $G=P(A_n)$, $P(B_n)$, or $P(D_4)$. Denote $\Gamma=G/Z(G)$ (recall that for $G=\Art(\widetilde A_{n-1})$ or $\Art(\widetilde C_{n-1})$, $\Gamma$ is isomorphic to $G$). Let $\varphi\in\Aut(\Gamma)$ and consider the group $\Gamma_{\varphi}$ as defined in Section~\ref{S:Sketch}. Then by Lemma~\ref{L:Embeds} and Proposition~\ref{prop:aut}, there exists a surface with punctures $\Sigma_\Gamma$ and a tower of inclusions of finite index subgroups of the extended mapping class group of $\Sigma_\Gamma$:
\[
\Gamma\leqslant\Gamma_\varphi\leqslant\Aut(\Gamma)\leqslant\Mod^\pm(\Sigma_\Gamma).
\]

Now recall that $\Mod^{\pm}(\Sigma_\Gamma)$, hence in particular its subgroup $\Gamma_{\varphi}$, acts by isometries on the curve complex $\C(\Sigma_\Gamma)$ of the surface~$\Sigma_\Gamma$. This curve complex is Gromov hyperbolic by the result of Masur--Minsky \cite[Th.\,1.1]{MM1} when $3g+p-4>0$, where $g$ denotes the genus of the surface $\Sigma_\Gamma$ and $p$ is the number of punctures. 
We claim that the action of $\Gamma_\varphi$ on $\C(\Sigma_\Gamma)$ is non-elementary. 
Indeed, it is known (see for instance~\cite{Bow08} and~\cite[Th.\,1.1]{Osi16}) that $\Mod^{\pm}(\Sigma_{\Gamma})$ contains 
infinitely many independent elements acting loxodromically on $\C(\Sigma_\Gamma)$. (Recall that an element acts loxodromically if its orbit has two accumulation points on the boundary 
and two such elements are independent if the sets of accumulation points of their orbits are disjoint.)
Notice that if an element acts loxodromically, so does any of its powers. Since $\Gamma_\varphi$ has finite index in $\Mod^{\pm}(\Sigma_{\Gamma})$, this implies that~$\Gamma_\varphi$ itself contains infinitely many independent elements acting loxodromically on $\C(\Sigma_\Gamma)$. 
Therefore, the action of $\Gamma_\varphi$ on $\C(\Sigma_\Gamma)$ is non-elementary and, by Proposition~\ref{P:Cosets} and Lemma~\ref{L:Delzant}, $R(\varphi)=\infty$. As~$\varphi$ was arbitrary, this implies that $\Gamma$ has property~$R_{\infty}$, and from Proposition~\ref{prop:epi} we deduce that $G$ has property~$R_{\infty}$. This finishes the proof of Theorem~\ref{T:Main}. 
\end{proof}

\begin{remark}
Given an Artin--Tits group $\Art(X)$,  there is another (in most cases, conjecturally) hyperbolic space with an action of the central quotient $\overline{\Art(X)}$. This is the \emph{graph of irreducible parabolic subgroups} $\mathcal C_{parab}(\Art(X))$ introduced in \cite{CGGMW} and also studied in~\cite{MorrisWright,CalvezCisneros}. Hyperbolicity of $\mathcal C_{parab}(\Art(X))$ was established for $X$ being $A_n$ ($n\geqslant 3$), $B_n$ ($n\geqslant 3$), $\widetilde A_n$ $(n\geqslant2)$, and $\widetilde C_n$ ($n\geqslant 2$), see \cite{CalvezCisneros}. Also, it can be shown using results from~\cite{Sor20} that $\mathcal C_{parab}(\Art(D_4))$ is isomorphic to the curve graph of the three times punctured torus, and hence it is also hyperbolic by Masur--Minsky's theorem. Given a non-inner automorphism $\varphi$ of $\overline{\Art(X)}$,  we found that in some cases, the action of $\overline{\Art(X)}$ extends to an action of $\overline{\Art(X)}_{\varphi}$ on $\mathcal C_{parab}(\Art(X))$; this is for instance the case when $\varphi$ is a parabolic-preserving automorphism. However, we also found that for some non-parabolic-preserving automorphisms of $\Art(\widetilde C_n)$, there is no action of $\Art(\widetilde C_n)_{\varphi}$ on $\mathcal C_{parab}(\Art(\widetilde C_n))$, so this approach does not yield a uniform proof of Theorem \ref{T:Main} as the one above.
\end{remark}


\begin{thebibliography}{CGGW19}
\bibitem[Bow08]{Bow08} Brian H. Bowditch, Tight geodesics in the curve complex, \textit{Invent. Math.} 171 (2008), no. 2,
281--300.

\bibitem[BS72]{BrieskornSaito} Egbert Brieskorn, Kyoji Saito, Artin-Gruppen und Coxeter-Gruppen, \textit{Invent. Math.} 17, (1972), 245--271.

\bibitem[CC21]{CalvezCisneros} Matthieu Calvez, Bruno Cisneros de la Cruz, Curve graphs for Artin--Tits groups of type $B$, $\widetilde A$ and $\widetilde C$ are hyperbolic. \textit{Trans. London Math. Soc.} 8 (1), (2021), 151--173.

\bibitem[CC05]{CC05} Ruth Charney, John Crisp, Automorphism groups of some affine and finite type Artin groups. \textit{Math. Res. Lett.} 12 (2005), no. 2--3, 321--333.

\bibitem[CP02]{CP02} John Crisp, Luis Paris, Artin groups of type B and D. Preprint (2002) \href{https://arxiv.org/abs/math/0210438v1}{arXiv:math/0210438v1}. 

\bibitem[CGGW19]{CGGMW} Mar\'ia Cumplido, Volker Gebhardt, Juan González‐Meneses and Bert Wiest, On parabolic subgroups of Artin‐-Tits groups of spherical type, \textit{Adv. Math.} 352 (2019) 572--610. 

\bibitem[CP22]{CP19} Mar\'ia Cumplido, Luis Paris, Commensurability in Artin groups of spherical type. \textit{Rev. Mat. Iberoam.} 38 (2022), no. 2, 503--526.

\bibitem[DGO21]{DGO21} Karel Dekimpe, Daciberg Lima Gon\c{c}alves, Oscar Ocampo, The $R_{\infty}$ property for pure Artin braid groups. \textit{Monatsh. Math.} 195 (2021), 15--33.

\bibitem[DS21]{DS21} Karel Dekimpe, Peter Senden, The $R_{\infty}$-property for right-angled Artin groups. \textit{Topology Appl.} 293 (2021), 107557.

\bibitem[FM12]{FM12} Benson Farb, Dan Margalit, \textit{A {P}rimer on {M}apping {C}lass {G}roups.}
Princeton Mathematical Series, 49. Princeton University Press, Princeton, NJ, 2012. xiv+472 pp.

\bibitem[Fel04]{Fel} Alexander Fel'shtyn, The Reidemeister number of any automorphism of a Gromov-hyperbolic group is infinite, 
\textit{Zap. Nauchn. Sem. S.-Peterburg. Otdel. Mat. Inst. Steklov. (POMI)} 279 (2001), Geom. i Topol. 6, 229--240, 250; reprinted in \textit{J. Math. Sci. (N.Y.)} 119 (2004), no. 1, 117--123.

\bibitem[FGD10]{FGD10} Alexander Fel'shtyn, Daciberg L. Gon\c{c}alves, Twisted conjugacy classes in symplectic groups, mapping class groups and braid groups. With an appendix written jointly with Fran\c{c}ois Dahmani. \textit{Geom. Dedicata} 146 (2010), 211--223.

\bibitem[FN16]{FN16} Alexander Fel'shtyn, Timur Nasybullov, The $R_{\infty}$ and $S_{\infty}$ properties for linear algebraic groups. \textit{J. Group Theory} 19 (2016), 901--921.

\bibitem[FT15]{FT15} Alexander Fel'shtyn, Evgenij Troitsky, Aspects of the property $R_{\infty}$, \textit{J. Group Theory} 18 (2015), 1021\nobreakdash--1034.

\bibitem[Juh10]{Juhasz} Arye Juhasz, Twisted conjugacy in certain Artin groups, Ischia Group Theory 2010, 175--195.

\bibitem[LL00]{LevittLustig} Gilbert Levitt, Martin Lustig, Most automorphisms of a hyperbolic group have very simple dynamics, \textit{Ann. Sci. Ec. Norm. Sup.}, 33, 4, (2000), 507--517.

\bibitem[MM99]{MM1} Howard Masur, Yair Minsky, Geometry of the complex of curves I: Hyperbolicity, \textit{Invent. Math.} 138 (1999), 103--149.

\bibitem[McCS20]{McCammondSulway} Jon McCammond, Robert Sulway, Artin groups of Euclidean type, \textit{Invent. Math.} 210, (2017), 231--282.

\bibitem[MW21]{MorrisWright} Rose Morris-Wright, Parabolic subgroups in FC-type Artin groups, \textit{J. Pure Appl. Algebra} 225 (2021), no.~1, 106468.

\bibitem[Osi16]{Osi16} Denis Osin, Acylindrically hyperbolic groups. \textit{Trans. Amer. Math. Soc.} 368 (2016), no. 2, 851--888.

\bibitem[Rom11]{Romankov} V. Roman'kov, Twisted conjugacy classes in nilpotent groups, \textit{J. Pure Appl. Algebra} 215 (2011), 664--671.

\bibitem[Sor20]{Sor20a} Ignat Soroko, Linearity of some low-complexity mapping class groups. \textit{Forum Math.} 32 (2020), no. 2, 279--286.

\bibitem[Sor21]{Sor20} Ignat Soroko, Artin groups of types $F_4$ and $H_4$ are not commensurable with that of type $D_4$. \textit{Topology Appl.}  300 (2021), 107770. 

\bibitem[TW11]{TW11} Jennifer Taback, Peter Wong, The geometry of twisted conjugacy classes in wreath products. \textit{Geometry, rigidity, and group actions}, 561--587, Chicago Lectures in Math., Univ. Chicago Press, Chicago, IL, 2011.

\bibitem[Tro19]{Tro19} Evgenij Troitsky, Reidemeister classes in some weakly branch groups. \textit{Russ. J. Math. Phys.}, 
 26 (2019), no. 1, 122--129.
\end{thebibliography}
\end{document}